\documentclass[12pt,a4paper]{amsart}
\usepackage{amsfonts,amsmath,amssymb,amsthm,cmtiup} 
\usepackage{mathrsfs}
\usepackage{enumerate}

\newtheorem{theorem}{Theorem}
\newtheorem*{theorem*}{Theorem}
\newtheorem{lemma}{Lemma} 
 
\newtheorem{proposition}{Proposition}

\theoremstyle{definition} 
\newtheorem{remark}{Remark}
\newtheorem*{definition*}{Definition}

\newcommand{\zero}{{\mathbf 0}}

\begin{document}

{
\renewcommand*{\thefootnote}{$\star$}

\title[Maximally almost periodic subgroups]{Maximally almost periodic 
subgroups\\ of elementary Abelian topological groups$^\star$} \footnotetext[0]{This work was 
financially supported by the Russian Science Foundation, grant~22-11-00075-P.} 

}

\author{Ol'ga Sipacheva}

\email{ovsipa@gmail.com}

\address{Department of General Topology and Geometry, Faculty of Mechanics and  Mathematics, 
M.~V.~Lomonosov Moscow State University, Leninskie Gory 1, Moscow, 199991 Russia}

\begin{abstract}
It is proved that any infinite elementary Abelian topological group has an infinite maximally almost 
periodic subgroup. 
\end{abstract}



\maketitle

This paper is concerned with elementary Abelian topological groups, i.e., Abelian topological groups in which all 
nonzero elements have the same prime order $p$. It is well known that, algebraically, any such group $G$ is a direct sum of copies of the cyclic 
group $\mathbb Z/p\mathbb Z$ of order $p$ (see, e.g., \cite[Theorem~8.5]{Fuchs}). Since $\mathbb Z/p\mathbb Z$ is 
a vector space over the field $\mathbb F_p$, it follows that so is $G$. Therefore, $G$ has a basis $E$, and every 
nonzero element of $G$ can be uniquely (up to the order of terms) represented as a linear combination of elements 
of $E$ with nonzero coefficients in $\mathbb Z/p\mathbb Z$; in other words, $G = \bigoplus_{e\in E}\langle 
e\rangle$, where $\langle e\rangle$ denotes the subgroup of $G$ generated by $e$ (this is the cyclic group of 
order $p$). Thus, on any elementary Abelian topological group $G$ with basis $E$, there exists the 
natural topology induced by the Tychonoff product topology of $\prod_{e\in E}\langle e\rangle$. In what follows, 
we refer to this topology as the \emph{product topology on~$G$ associated with~$E$}. We assume all topological 
groups under consideration to be Hausdorff. 

Although the algebraic structure an elementary Abelian group is simple, its topology may be very complex and 
strongly differ from a product topology. For example, it may happen that all continuous characters 
of such a group (that is, continuous homomorphisms to the circle group $\mathbb T$) are 
trivial~\cite[Example~5]{ZP}. Recall that the intersection of the kernels of continuous characters of a 
topological group $G$ is called the \emph{von~Neumann kernel} of $G$ and denoted by $n(G)$; a group $G$ 
with $n(G)= G$ is said to be \emph{minimally almost periodic} and a group $G$ with trivial $n(G)$ is said to be 
\emph{maximally almost periodic}~\cite{Neumann-Wigner}. 

\begin{remark}
\label{r1}
If an elementaty Abelian topological group $G$ has a basis $E$ such that the product topology associated 
with $E$ is coarser than the original topology of $G$ (or coincides with it), then $G$ is maximally almost 
periodic. Indeed, in this case, all projections $\pi_{e_0}\colon \bigoplus_{e\in E}\langle 
e\rangle \to \langle e_0\rangle$ onto the finite discrete factors $\langle e_0\rangle$, $e_0\in E$, 
are continuous (since they are continuous with respect to the Tychonoff product topology), and these are 
characters of~$G$, because every finite cyclic group $\langle e\rangle$ with the discrete topology is embedded in $\mathbb T$ as a 
topological subgroup. Clearly, these characters separate the points of $G$ and hence the intersection of their 
kernels is trivial. 

Note also that an elementary Abelian topological group $G$ of exponent $p$ is maximally almost periodic if and only if 
the intersection of open subgroups of finite index in $G$ is trivial, because the image of any nontrivial 
continuous character of $G$ is the discrete cyclic group 
$\mathbb Z/p\mathbb Z$ and its kernel is an open subgroup of $G$ of index~$p$. On the other hand, every 
open subgroup of finite index in $G$ is an intersection of open subgroups of index $p$, each of which is the 
kernel of a continuous character. 
\end{remark}

Example~5 of \cite{ZP} proves the existence of a Hausdorff minimally almost periodic elementary 
Abelian topological group of any prime exponent. However, such groups are never hereditarily minimally almost 
periodic: in this paper, we show that any infinite elementary Abelian topological group 
contains an infinite maximally almost periodic subgroup whose topology is finer than the 
product topology associated with some basis of this subgroup.  

In what follows, we use $\mathbb N$ and $\mathbb Z$ to denote the sets of positive integers and integers, 
respectively. The subgroup of a group $G$ generated by a set $A\subset G$ is denoted by $\langle A\rangle$. 
We assume that $\langle \varnothing\rangle$ is the trivial subgroup. Recall that a \emph{network} of a 
topological space $X$ is a family $\mathscr N$ of subsets of $X$ such that any open set in $X$ is the union 
of some elements of $\mathscr N$ (a network differs from a base in that its elements are not necessarily open). 
All notions related to topological and metric 
spaces can be found in~\cite{Engelking} and to topological groups, in~\cite{AT}. 

A \emph{norm} on a group $G$ with neutral element $e$ is a function $\lVert\boldsymbol\cdot\rVert\colon G\to 
\mathbb R$ with the following properties: 
\begin{enumerate} 
\item $\|g\|=0$ if and only if $g=e$; 
\item $\|g^{-1}\| = \|g\|$; 
\item $\|gh\|\le \|g\|+\|h\|$ (the triangle inequality for norms). 
\end{enumerate} 
We say that a topology of a topological group $G$ is \emph{generated by a norm} $\lVert \boldsymbol\cdot \rVert$ 
if the sets $\{x\in G: \|x\|< \varepsilon\}$, $\varepsilon > 0$, form a neighborhood base at $e$ for the topology 
of~$G$. 

\begin{definition*}
We say that a subset $X$ of a group $G$ with neutral element $e$ is \emph{independent} if $\langle A\rangle \cap 
\langle X\setminus A\rangle=\{e\}$ for any $A\subset X$. 
\end{definition*}

It is easy to see that, for an Abelian group $G$ with zero element $\zero$, a set $X\subset G$ is independent if, 
given any 
$n\in \mathbb N$, any pairwise distinct $x_1,\dots, x_n\in X$, and any integer $k_1,\dots, k_n$, we have 
$k_1 x_1=\dots= k_n x_n=\zero$ whenever $k_1 x_1+ \dots + k_n x_n=\zero$.   

\begin{definition*}
We say that a subset $X$ of a group $G$ is a \emph{basis} of $G$ if $X$ is independent and $\langle X\rangle =G$.
\end{definition*}

Clearly, given  is an elementary Abelian group $G$ of exponent $p$, a set $E\subset G$ is a basis of $G$ in the 
sense of this definition if and only if $E$ is a basis of $G$ treated as a vector space over the 
field~$\mathbb F_p$. 

Not every group has a basis, but when a basis exists, it is almost never unique. As mentioned above, any Abelian 
group $G$ of prime exponent does have a basis $E$ and every element $g$ of such a group can be written as a 
linear combination of elements of $E$ with nonzero coefficients in a unique way (up to permutation of terms). We 
refer to this linear combination as a \emph{decomposition} of $g$ in~$E$. 

\begin{definition*}
A subset $X$ of an Abelian topological group $G$ with zero element $\zero$ is said to be 
\emph{topologically independent} if, given any $n\in \mathbb N$, any pairwise distinct $x_1,\dots, 
x_n\in X$, any $k_1,\dots, k_n\in \mathbb Z$, and any neighborhood $U$ of $\zero$, there exists a neighborhood 
$V$ of $\zero$ such that $k_1 x_1, \dots, k_n x_n\in U$ whenever $k_1 x_1 +\dots +k_n x_n\in V$.
\end{definition*}

Note that any Tychonoff space $X$ is a basis of the free Abelian and free Boolean topological groups 
$A(X)$ and $B(X)$ (see, e.g., \cite{axioms}), but it is never topologically independent in these groups (unless $X$ is 
discrete). Moreover, no subspace $Y$ of $X$ with nondiscrete closure is topologically independent in 
$A(X)$ and $B(X)$. Indeed, let 
$G$ be any of the groups $A(X)$, and $B(X)$; we denote the zero element of $G$ by $\zero$. The subgroup $H$ of 
$G$ consisting of all words of even length on the alphabet $X$ is an open neighborhood of $\zero$, because this 
is the kernel of the continuous homomorphism $H\to \mathbb Z/2\mathbb Z$ extending the function 
$X\to \{0,1\}= \mathbb Z/2\mathbb Z$ defined by setting $x\mapsto 1$ for all $x\in X$. Let $U = H$, and let $y$ 
be any limit point of $Y$ in $X$. If $Y$ is topologically independent, then there exists a 
neighborhood $V$ of $\zero$ such that $u,v\in U$ for any $u,v\in Y$ satisfying the condition $u-v\in V$.  
Let $W$ be a neighborhood of $\zero$ for which $W-W\subset V$; then $(y+W)\cap X$ is a neighborhood of $y$ 
in $X$. For any distinct $x,z\in (y+W)\cap X$, we have $x-z\in W-W\subset V$, but $x$ and $-z$ do not belong 
to $U$, being words of odd length 1 on the alphabet $X$. Hence $Y$ is not topologically independent. 

\begin{proposition}
\label{p1}
Any countable metrizable elementary Abelian topological group contains an infinite topologically 
independent set. 
\end{proposition}

\begin{proof}
Let $G$ be a countable metrizable elementary Abelian topological group of prime exponent $p$ 
with zero element $\zero$, and let $E=\{e_1, e_2, \dots\}$ be any basis of $G$. For convenience, we will treat 
$G$ as a vector space over 
$\mathbb F_p$ and its elements as linear combinations of elements of~$E$ with coefficients in~$\mathbb F_p$. 

Since the group $G$ is first-countable, its topology is generated by a norm $\lVert\boldsymbol\cdot\rVert$ (see, 
e.g., \cite{Graev50}). Consider the set $E'=\{e'_1, e'_2, \dots\}$ defined recursively as follows: 
\begin{itemize} 
\item 
$e'_1= e_1$; 
\item 
if $n\in \mathbb N$ and $e'_1, e'_2, \dots, e'_{n}$ are already defined, then $e'_{n+1}$ is 
a linear combination $\lambda_1 e'_1 +\lambda_2 e'_2 + \dots+ \lambda_n e'_{n}+\lambda_{n+1} 
e_{n+1}$, where $\lambda_{n+1}\ne 0$, of minimum norm (if there are 
several linear combinations of minimum norm, then we take any of them). 
\end{itemize}

It is easy to see that $E'$ has the following properties:
\begin{enumerate}[{\rm(i)}]
\item
$\langle \{e'_1, \dots, e'_n\}\rangle =\langle \{e_1, \dots, e_n\}\rangle$ for each $n$;
\item
$E'$ is a basis in $G$;
\item
if $n_1<\dots <n_k$, then $\|e'_{n_k}\|\le \|\lambda_1 e'_{n_1}+\dots +\lambda_k e'_{n_k}\|$ for any $\lambda_1, 
\dots, \lambda_k\in \mathbb F_p$. 
\end{enumerate} 

Indeed, (i) is proved by simple induction: for every $n$, $e'_n$ can be represented as a linear combination 
of $e_1, \dots, e_n$ by construction. On the other hand, we have 
$$
e'_n=\lambda_1 e'_1+ \lambda_2e'_2+ \dots +\lambda_{n-1}e'_{n-1}+\lambda_n e_n,
$$ 
where $\lambda_n\ne 0$. Therefore, $(p-\lambda_n) e_n= 
\lambda_1 e'_1+ \lambda_2e'_2+ \dots +\lambda_{n-1}e'_{n-1}+(p-1)e'_n$ and hence 
$e_n$ can be represented as a linear combination of $e'_1, \dots, e'_{n}$. This easily implies~(i).

The set $E'$ is linearly independent, because the decomposition of every $e'_n$ in $E$  
contains $e_n$ with nonzero coefficient, which does not occur in the decompositions of $e'_k$ with $k<n$. Since 
$E$ is a basis of $G$, it follows from (i) that $E'$ spans $G$. This implies~(ii).

Property~(iii) follows directly from the definition of $e'_{n_k}$, because  
$\lambda_1 e'_{n_1}+\dots +\lambda_k e'_{n_k}$ equals a linear 
combination of $e_1, e_2, \dots, e_{n_k}$ (since so does each $e'_{n_i}$) and the coefficient of $e_{n_k}$ 
in this linear combination is nonzero (since $e_{n_k}$ occurs in the decomposition of $e'_{n_k}$ in 
$E$ with nonzero coefficient and does not occur in the decompositions of $e'_k$ with $k<n$). 

\begin{lemma}
\label{l1}
Let $w=\lambda_1 e'_{i_1} + \lambda_2 e'_{i_2} +\dots + \lambda_n e'_{i_n}$\textup, 
where $n, i_1, \dots, i_n\in \mathbb N$\textup, $i_1< i_2< \dots < i_n$\textup, and $\lambda_1, \dots, 
\lambda_n\in \mathbb F_p\setminus \{0\}$\textup. Then 
$$ 
\|\mu e'_{i_{n-k}}\|\le p^k\cdot 2^k\cdot \|w\|\qquad \text{for each}\quad k=0, \dots, n-1\quad \text{and 
any}\quad \mu\in \mathbb F_p. 
$$ 
\end{lemma}

\begin{proof}
We argue by induction on $k$. Note that, for any $\mu \in \mathbb F_p$ and $x\in G$, we have $\|\mu x\|\le p\cdot 
\|x\|$ by virtue of the triangle inequality for norms. Thus, the required inequality for $k=0$ follows from property (iii) of $E'$. Suppose that 
$0<l<n$ and the inequality holds for any $k<l$ and $\mu\in \mathbb F_p$. Applying the triangle inequality for 
norms, we obtain 
\begin{multline*}
\|(p-\lambda_{n-l+1})e'_{i_{n-l+1}}+(p-\lambda_{n-l+2})e'_{i_{n-l+2}}+\dots +(p-\lambda_{n})e'_{i_{n}}\| \\
\le p^{l-1}\cdot 2^{l-1}\cdot\|w\|+ p^{l-2}\cdot 2^{l-2}\cdot\|w\|+\dots+p^0\cdot 2^0\cdot\|w\|\\
\le p^{l-1}\cdot (2^{l-1}+2^{l-2}+\dots+1)\cdot\|w\|= p^{l-1}\cdot (2^l-1)\cdot \|w\|. 
\end{multline*}
Noting that 
\begin{multline*}
\lambda_1 e'_{i_{1}}+\lambda_2 e'_{i_{2}}+\dots +\lambda_l e'_{i_{l}}  \\ = w + 
(p-\lambda_{n-l+1})e'_{i_{n-l+1}}+(p-\lambda_{n-l+2})e'_{i_{n-l+2}}+\dots +(p-\lambda_{n})e'_{i_{n}}
\end{multline*}
and 
applying the triangle inequality once more, we arrive at 
$$
\|\lambda_1 e'_{i_{1}}+\lambda_2 e'_{i_{2}}+\dots +\lambda_l e'_{i_{l}}\|\le 
\|w\|+ p^{l-1}\cdot (2^l-1)\cdot \|w\| \le p^{l-1}\cdot 2^l\cdot\|w\|. 
$$
By property (iii) we have $\|e'_{i_{l}}\|\le p^{l-1}\cdot 2^l\cdot\|w\|$ and $\|\mu e'_{i_{l}}\|\le p^l \cdot 
2^l\cdot\|w\|$. 
\end{proof}

\begin{lemma}
\label{l2}
Each Cauchy sequence of pairwise distinct elements of $E'$ converges to~$\zero$.
\end{lemma}

\begin{proof}
Let $(e'_{n_k})_{k\in \mathbb N}$ be a Cauchy sequence of pairwise distinct elements. 
For any $\varepsilon >0$, there exists an $N\in \mathbb N$ such that 
$\|e'_{n_{k}}-e'_{n_N}\|< \varepsilon$ for any $k\ge N$. Without loss of generality, we may assume that 
$n_k>n_N$ for all $k> N$. But if $n_k > n_N$, then 
$\|e'_{n_k}\|\le \|e'_{n_k}-e'_{n_N}\|$ by property~(iii) of $E'$. Therefore, 
$\|e'_{n_k}\|<\varepsilon$ for all~$k > N$.
\end{proof}

Let $\max\colon G\to \{0\}\cup \mathbb N$ be the function defined by $\max(\zero) = 0$ and 
$$
\max (\lambda_1 e'_{n_1} + \dots + \lambda_k e'_{n_k}) = n_k
\quad \text{for $n_1< \dots < n_k$ and $\lambda_1, \dots, \lambda_k\in \mathbb F_p\setminus \{0\}$}.
$$
In particular, $\max(e'_n)=n$ for each $e'_n$.

If $G$ is discrete, then any independent set (in particular, any basis) in $G$ is topologically independent. 
Suppose that $G$ is nondiscrete. Then there exists a sequence $(g_n)_{n\in \mathbb N}$ of pairwise distinct 
elements of $g$ converging to $\zero$. We may assume without loss of generality that the sequence 
$\bigl(\max(g_n)\bigr)_{n\in \mathbb N}$ strictly increases and that $\|g_n\|<1/(4p)^n$. For each $n\in \mathbb 
N$, we set $k_n = \max (g_n)$ and $a_n=e'_{k_n}$. Then $\|a_n\|<1/(4p)^n$ for $n\in \mathbb N$ by property 
(iii). Let $\varepsilon >0$. Choose a positive integer $l$ for which $1/2^{l-1}<\varepsilon$ and consider 
any linear combination $w=\lambda_1 a_{n_1}+ \dots + \lambda_m a_{n_m}$, where $n_1<\dots < n_m$ 
(and hence $n_i\ge i$ for $i\le m$). We have $w = w'+ w''$, where 
$w'=\lambda_1 a_{n_1}+ \dots + \lambda_l a_{n_l}$ and $w''=\lambda_{l+1} a_{n_{l+1}}+ \dots + \lambda_m 
a_{n_m}$ (if $m\le l$, then $w''=\mathbf 0$ and possibly $\lambda_i=0$ for some $i\le l$). 
Clearly,  $\|w'\|\le \|w\|+\|(p-1)w''\|$ and 
$$ 
\|(p-1)w''\|\le p^2\cdot (\|a_{l+1}\|+\dots+\|a_m\|) < 1/(4p)^{l-1}, 
$$ 
whence $\|w'\|\le \|w\|+ 1/(4p)^{l-1}$. According to Lemma~\ref{l1}, $\|a_{n_i}\|\le (2p)^l\cdot 
\|w'\|$ for each $i\le l$. Therefore, if $\|w\|<1/(4p)^l$, then $\|a_{n_i}\|\le (2p)^l/(4p)^l + 
1/(4p)^{l-1} < 1/2^{l-1}$ for each $i\le l$. Moreover, by the definition of the sequence $(a_n)_{n\in 
\mathbb N}$, we have $\|a_{n_i}\|<1/(4p)^i<1/2^{l-1}$ for all $i>l$, $i\le m$. Thus, if 
$\|w\|<1/(4p)^l$, then $\|a_{n_i}\|<\varepsilon$ for all $i\le m$. 
This means that the set $\{a_n: n\in \mathbb N\}$ is topologically independent.
\end{proof}

\begin{theorem}
\label{t1}
Any infinite elementary Abelian topological group contains an infinite maximally almost periodic subgroup. 
\end{theorem}

\begin{proof}
Let $G$ be an infinite elementary Abelian topological group of prime exponent $p$ with zero element $\zero$, 
and let $H$ be 
an arbitrary countable subgroup of $G$. Any countable topological space has a countable network (consisting of 
singletons), and 
any Hausdorff topological group with a countable network admits a not strictly coarser second-countable Hausdorff group 
topology \cite[Theorem~2]{Arhangelskii79}. Let $\tau$ be a second-countable Hausdorff group topology on $H$ 
coarser than the induced topology. By Proposition~\ref{p1} \,$(H,\tau)$ contains an infinite topologically 
independent set $A=\{a_n: n\in \mathbb N\}$, where $a_i\ne a_j$ for $i\ne j$. Since $A$ is independent, it 
follows that the subgroup $\langle A\rangle$ of $G$ generated by $A$ is isomorphic to the direct sum 
$\bigoplus_{n\in \mathbb N} \langle a_n\rangle\subset \prod_{n\in \mathbb N}\langle a_n\rangle$ of the finite 
groups $\langle a_n\rangle$. According to \cite[Proposition~4.7]{DSS}, the product topology on $\langle A\rangle$ 
associated with $A$ is coarser than the topology induced from $(H,\tau)$; hence it is coarser than that induced 
from~$G$. According to Remark~\ref{r1}, $\langle A\rangle$ is maximally almost periodic. 
\end{proof}

As mentioned, any elementary Abelian topological group $G$ of exponent $p$ carries a natural structure of a 
topological vector space over the field $\mathbb F_p$ and the projection onto any one-dimensional subspace 
of this space is a character of $G$, however, not necessarily continuous; moreover, it may happen that 
no notrivial character of $G$ and hence no projection of the corresponding vector space onto a one-dimensional 
subspace is is continuous \cite[Example~5]{ZP}. Nevertheless, the following theorem holds. 

\begin{theorem}
\label{t2}
Any infinite-dimensional topological vector space $V$ over the field $\mathbb F_p$, where $p$ is any prime, has 
an infinite-dimensional subspace $W$ whose induced topology is finer than the product topology associated with 
some basis $E$ of $W$; therefore, all projections of $W$ onto its one-dimensional subspaces $\langle e\rangle$, 
$e\in E$, are continuous. 
\end{theorem}

\begin{proof}
Consider the additive group $G$ of $V$. Arguing as in the proof of Theorem~\ref{t1}, 
we take any countable subgroup $H$ of $G$ and find a Hausdorff second-countable group topology $\tau$ on $H$ 
coarser than that induced from $G$ and an infinite topologically independent set $E$ in $(H, \tau)$. Clearly, 
$E$ is a linearly independent subset of the vector space $V$ and the subgroup $\langle E\rangle$ of $H$ is 
the vector subspace of $V$ with basis $E$. Let us denote this subspace by $W$. Just as in the proof of 
Theorem~\ref{t1}, it follows from Proposition~4.7 of \cite{DSS} that the product topology on $W$ associated 
with $E$ is coarser than that induced from~$G$ (i.e., from $V$). All projections $\pi_e\colon W\to 
\langle e\rangle$, $e\in E$, are continuous with respect to the product topology, and each one-dimensional 
subspace $\langle e\rangle$ is finite and hence discrete in both the topology induced from $V$ and the product 
topology. Therefore, all projections $\pi_e$ are continuous with respect to the induced topology on~$W$. 
\end{proof}

\end{document}